\documentclass[letter,english]{article}

\usepackage[utf8]{inputenc}
\usepackage[T1]{fontenc}
\usepackage{babel}
\usepackage[babel]{csquotes}

\usepackage{amsmath}
\usepackage{amsfonts}
\usepackage{amsthm}
\usepackage{bbm}
\usepackage{amssymb}
\usepackage{IEEEtrantools}
\usepackage{url}

\newtheorem{theorem}{Theorem}[section]
\newtheorem{proposition}[theorem]{Proposition}
\newtheorem{lemma}[theorem]{Lemma}
\newtheorem{corollary}[theorem]{Corollary}
\theoremstyle{definition}
\newtheorem{definition}[theorem]{Definition}
\newtheorem{example}{Example}
\newtheorem{remark}{Remark}

\DeclareMathOperator{\Ad}{Ad}

\DeclareMathOperator{\Stab}{Stab}

\DeclareMathOperator{\card}{card}
\DeclareMathOperator{\Span}{Span}

\DeclareMathOperator{\Lie}{Lie}

\newcommand{\R}{\mathbb{R}}

\newcommand{\g}{\mathfrak{g}}
\newcommand{\h}{\mathfrak{h}}

\title{Borelian subgroups of simple Lie groups}
\author{Nicolas de Saxcé
\thanks{The author is supported by ERC AdG grant 267259, and acknowledges partial support from ANR-12-BS01-0011 (CAESAR).}
}

\begin{document}

\maketitle

\begin{abstract}
We prove that in a simple real Lie group, there is no Borel measurable dense subgroup of intermediate Hausdorff dimension.
\end{abstract}

\section{Introduction}

The main purpose of the present paper is to prove the following theorem.

\begin{theorem}\label{nointermediatei}
Let $G$ be a connected simple real Lie group endowed with a Riemannian metric. There is no Borel measurable dense subgroup of $G$ with Hausdorff dimension strictly between $0$ and $\dim G$.
\end{theorem}

For the group $SU(2)$, Theorem~\ref{nointermediatei} was proved by Lindenstrauss and Saxcé \cite{lindenstrausssaxcesu2}.
In contrast, it is shown in \cite{saxcenilpotent} that in a connected nilpotent Lie group $G$ there exist dense Borel measurable subgroups of arbitrary dimension between $0$ and $\dim G$.

The study of subgroups of Lie groups with intermediate Hausdorff dimension started with the work of Erd\H{o}s and Volkmann \cite{erdosvolkmann}, who constructed additive subgroups of the real line with arbitrary Hausdorff dimension between $0$ and $1$, and conjectured that any Borel subring of the reals has Hausdorff dimension $0$. This conjecture was settled by Edgar and Miller \cite{edgarmiller} in 2002, and shortly afterwards, Bourgain \cite{bourgainringconjecture,bourgainprojection} provided an independent and more quantitative solution.

The proof of Theorem~\ref{nointermediatei} given in this paper follows the strategy of \enquote{discretization} used by Bourgain in its solution to the Erd\H{o}s-Volkmann Conjecture, and also yields the following more precise theorem.

\begin{theorem}\label{dimineqi}
Let $G$ be a connected simple real Lie group endowed with a Riemannian metric.
There exists a neighborhood $U$ of the identity in $G$ and a positive integer $k$ such that for all $\sigma>0$, there exists $\epsilon=\epsilon(\sigma)>0$ such that the following holds.\\
Suppose $A$ is a Borel subset of $U$ generating a dense subgroup of $G$ and with Hausdorff dimension $\dim_H A \in [\sigma,\dim G-\sigma]$, then
$$\dim_H A^k \geq \dim_H A + \epsilon,$$
where $A^k$ denotes the set of all elements of $G$ that can be written as products of $k$ elements of $A$.
\end{theorem}

It should be noted that the assumption that the set $A$ is Borel measurable cannot be omitted. Indeed, Davies~\cite{daviesunpublished} showed that there exist non-Borel subfields of the real line of arbitrary Hausdorff dimension (see also~\cite{falconerlargeintersection}); it is then easy to check that if $F$ is a subfield of $\R$ of Hausdorff dimension $\alpha$, then the subgroup $SL(2,F)$ in $SL(2,\R)$ has Hausdorff dimension $3\alpha$.

\paragraph{}

The idea of \enquote{discretization} is to translate problems about Hausdorff dimension into combinatorial problems about covering numbers of sets by balls of some small fixed radius $\delta$.
For that, Katz and Tao \cite{katztao} introduced the notion of $(\sigma,\epsilon)$-set at scale $\delta$, which is the natural discretized analog of sets of Hausdorff dimension $\sigma$.
The study of Hausdorff dimension of product sets then consists into three steps:
first, one proves a combinatorial statement about covering numbers of $(\sigma,\epsilon)$-sets at scale $\delta$,
then one deduces from it a flattening statement for measures,
and finally, using Frostman's Lemma, on derives an inequality on Hausdorff dimensions.

In the proof of Theorems~\ref{nointermediatei} and \ref{dimineqi}, the combinatorial part is based on a discretized Product Theorem for simple Lie groups \cite[Theorem~1.1]{saxceproducttheorem}.
A key point in this combinatorial analysis is to understand the set $\Xi$ of \enquote{troublemakers} of a subset $A$ in $G$.
Roughly speaking, those are the elements $\xi$ such that there exist large subsets $A'$ and $B'$ in $A$ such that the product set $A'\xi B'$ is not much larger than $A$.
By controlling the structure of approximate subgroups in $G$, we will show that if $A$ is a $(\sigma,\epsilon)$-set at scale $\delta$, then the set $\Xi$ is included in a union of few neighborhoods of cosets of closed subgroups of $G$.
This observation will allow us to prove the expansion statement needed to derive flattening of measures.

\paragraph{}

The plan of the paper is as follows.
In Section~\ref{section:subgroups}, we investigate the structure of approximate subgroups of $G$ and derive some elementary lemmas about subgroup chunks.
Section~\ref{section:troublemakers} is devoted to the proof of the combinatorial discretized version of Theorem~\ref{dimineqi}.
Finally, in Section~\ref{section:flattening}, we prove a Flattening Lemma for Frostman measures, and carry out the applications to Hausdorff dimension of product sets.

\paragraph{Acknowledgements}
I am very grateful to Yves Benoist for many helpful and motivating discussions, for his precious comments on a previous version of this manuscript, and above all, for his enthusiasm for this problem.\\
I also thank Emmanuel Breuillard, with whom this problem was raised, during my doctoral thesis under his supervision, and Elon Lindenstrauss for interesting discussions.

\section{Approximate subgroups and subgroup chunks}
\label{section:subgroups}

\subsection{Controlling approximate subgroups}

We start by recalling some elementary facts from additive combinatorics. If $A$ and $B$ are subsets of a group $G$, we denote by $AB$ the product set of $A$ and $B$, i.e.
$$AB = \{ab \,;\, a\in A,\ b\in B\}.$$
Similarly, for $k\geq 1$, $A^k$ denotes the set of elements that can be written as the product of $k$ elements of $A$.
An important definition for us will be that of an approximate subgroup, due to Tao~\cite{taoestimates}.

\begin{definition}
Let $G$ be a metric group, and $K\geq 1$ a parameter. A $K$-approximate subgroup of $G$ is a subset of $G$ satisfying
\begin{itemize}
\item $A$ is symmetric and contains the identity.
\item There exists a finite set $X$ of cardinality at most $K$ such that $AA\subset XA$.
\end{itemize}
\end{definition}

In this paper, $G$ will always denote a connected simple Lie group, endowed with a left-invariant Riemannian metric.
If $A$ is a bounded subset of $G$, and $\delta>0$ is some small scale, we denote by $N(A,\delta)$ the minimal number of balls of radius $\delta$ needed to cover $A$.

For the application to the study of Hausdorff dimension of product sets, the following definition, due to Katz and Tao~\cite{katztao} is appropriate.

\begin{definition}
Let $G$ be a real Lie group of dimension $d$. Given $\sigma\in(0,d)$ and $\epsilon>0$, we say that a subset $A$ in $G$ is a \emph{$(\sigma,\epsilon)$-set at scale $\delta$} if it satisfies
\begin{enumerate}
\item $N(A,\delta) \leq \delta^{-\sigma-\epsilon}$
\item For all $\rho\geq\delta$, for all $x$ in $G$, $N(A\cap B(x,\rho),\delta) \leq \rho^\sigma\delta^{-\epsilon}N(A,\delta)$.
\end{enumerate}
\end{definition}

\begin{remark}
One should think of $(\sigma,\epsilon)$-sets at scale $\delta$ as sets of Hausdorff dimension $\sigma$ discretized at scale $\delta$. The parameter $\epsilon$ quantifies what we lose in the discretization process.
\end{remark}

\begin{example}
For $\sigma=\frac{\log 2}{\log 3}$ and any $\epsilon>0$, the usual triadic Cantor set is a $(\sigma,\epsilon)$-set at scale $\delta$ for all $\delta$ sufficiently small.
\end{example}

Given a connected simple Lie group $G$ endowed with a Riemannian metric,
we want to describe the structure of $(\sigma,\epsilon)$-sets in $G$ that are also $\delta^{-\epsilon}$-approximate subgroups.
For that purpose, we make the following definition.

\begin{definition}
Let $G$ be a Lie group, and fix $O$ a neighborhood of $0$ in the Lie algebra $\g$ on which the exponential map is injective. Given a symmetric neighborhood $U$ of the identity such that $U\subset\exp O$, we define a \emph{subgroup chunk in $U$} to be a set of the form
$U\cap\exp(O\cap\h)$, for some Lie subalgebra $\h<\g$.\\
Similarly, a \emph{coset chunk in $U$} is a set of the form $U\cap g\exp(O\cap\h)$, for some Lie subalgebra $\h<\g$ and some element $g$ in $U$.
\end{definition}

Throughout the paper, if $X$ is any subset of $G$ and $\rho$ some positive number, $X^{(\rho)}$ denotes the $\rho$-neighborhood of $X$ in $G$, i.e.
$$X^{(\rho)} = \{x\in G \,|\, d(x,X)\leq\rho\}.$$
What will allow us to control approximate subgroups with subgroup chunks is the following discretized Product Theorem \cite[Theorem~1.1]{saxceproducttheorem}.

\begin{theorem}[Product Theorem]\label{producttheorem}
Let $G$ be a simple real Lie group of dimension $d$. There exists a neighborhood $U$ of the identity in $G$ such that the following holds.\\
Given $\sigma\in (0,d)$, there exists $\tau=\tau(\sigma)>0$ and $\epsilon_0=\epsilon_0(\sigma)>0$ such that, for all $\epsilon\in(0,\epsilon_0)$, for all $\delta>0$ sufficiently small, if $A\subset U$ is a set satisfying
\begin{enumerate}
\item $N(A,\delta)\leq\delta^{-\sigma-\epsilon}$
\item $\forall \rho\geq\delta$, $N(A,\rho) \geq \delta^{\epsilon}\rho^{-\sigma}$
\item $N(AAA,\delta)\leq\delta^{-\epsilon} N(A,\delta)$
\end{enumerate}
then
there exists a closed connected subgroup $H\subset G$, such that
$$A \subset H^{(\delta^{\tau})}.$$
Moreover, $\tau$ and $\epsilon_0$ remain bounded away from zero when $\sigma$ varies in a compact subset of $(0,d)$.
\end{theorem}

\begin{remark}\label{se}
Note that if $A$ is a $(\sigma,\epsilon)$-set at scale $\delta$, then it necessarily satisfies the first two conditions of the Product Theorem.
\end{remark}

\begin{remark}\label{maxchunk}
In the conclusion of the Product Theorem, we may of course assume that the closed connected subgroup $H$ is maximal. If this is the case, then we know \cite[Proposition~2.1]{saxceproducttheorem} that, provided $U$ has been chosen small enough, $H\cap U$ is just the subgroup chunk in $U$ with Lie algebra $\h=\Lie H$.
\end{remark}

Given a $(\sigma,\epsilon)$-set $\widetilde{H}$ that is also a $\delta^{-\epsilon}$-approximate subgroup, we know from the Product Theorem that $\widetilde{H}$ is included in a small neighborhood of a proper subgroup chunk.
The purpose of the following lemma is to allow us to choose the subgroup chunk $H'$ of minimal dimension that can control $\widetilde{H}$.

\begin{lemma}\label{niceh}
Let $G$ be a simple Lie group of dimension $d$. There exists a neighborhood $U$ of the identity in $G$ such that the following holds.\\
Given $\sigma\in(0,d)$ and $b\in(0,1)$, there exist constants $K_\ell$ and $\tau_\ell=\tau_\ell(\sigma,b)>0$, for $\ell\in\{1,\dots,d-1\}$, and $\epsilon_0=\epsilon_0(\sigma)>0$ such that the following holds for any $\epsilon\in(0,\epsilon_0)$ and any $\delta>0$ small enough.\\
Suppose $\widetilde{H}\subset U$ is a $(\sigma,\epsilon)$-set at scale $\delta$ and a $\delta^{-\epsilon}$-approximate subgroup.\\
There exists $\ell$ in $\{1,\dots,d-1\}$, a subgroup chunk $H'$ in $U^4$
of dimension $\ell$ and a subset $\widetilde{H}'\subset H'^{(\tau_\ell)}$, such that:
\begin{enumerate}
\item There is a finite set $X$ of cardinality at most $\delta^{-K_\ell\epsilon}$ such that\\
$\widetilde{H}\subset X\widetilde{H}'\cap \widetilde{H}'X$.
\item If $D$ is any coset chunk in $U$ such that $\dim D<\ell$, then\\
$N(\widetilde{H}'\cap D^{(\delta^{b\tau_\ell})},\delta) \leq \delta^{8K_\ell\epsilon}N(\widetilde{H},\delta)$.
\end{enumerate}
\end{lemma}
\begin{proof}
Choose a symmetric neighborhood $U$, and parameters $\tau$ and $\epsilon_0$ for which the Product Theorem~\ref{producttheorem} holds. We also assume that $U^4$ is still an exponential neighborhood of the identity in $G$.\\
Then, let $K_\ell=2\cdot 10^{d-\ell}$ and $\tau_\ell=(\frac{b}{3})^{d-\ell}\tau$, and choose $\ell$ maximal such that for any coset chunk $D$ of dimension less than $\ell$,
$$N(\widetilde{H}\cap D^{(\delta^{\frac{b\tau_\ell}{2}})},\delta) \leq \delta^{9K_\ell\epsilon}N(\widetilde{H},\delta).$$
By the Product Theorem~\ref{producttheorem} and Remark~\ref{maxchunk}, there exists a proper subgroup chunk $H_0$ in $U$ such that $\widetilde{H}\subset H_0^{(\delta^{\tau})}$. This shows that $\ell\leq d-1$.\\
On the other hand, coset chunks of dimension $0$ are just points, and using that $\widetilde{H}$ is a $(\sigma,\epsilon)$-set at scale $\delta$, we see that, provided $\epsilon$ is sufficiently small, for any $x$ in $G$, one has,
$$N(\widetilde{H}\cap B(x,\delta^{\frac{b\tau_1}{2}}),\delta)
\leq \delta^{\sigma \frac{b\tau_1}{2}-\epsilon}N(\widetilde{H},\delta)
\leq \delta^{9K_1\epsilon}N(\widetilde{H},\delta).$$
So we also have $\ell\geq 1$.\\
By maximality of $\ell$, there exists an $\ell$-dimensional coset chunk $C$ in $U$ such that
$$N(\widetilde{H}\cap C^{(\delta^{\frac{b\tau_{\ell+1}}{2}})},\delta)
= N(\widetilde{H}\cap C^{(\delta^{\frac{3\tau_\ell}{2}})},\delta)
\geq \delta^{9K_{\ell+1}\epsilon}N(\widetilde{H},\delta).$$
Writing $C=gH'$ for some subgroup chunk $H'$ and some element $g$ in $U$, one readily sees that, for some constant $L$ depending only on $U$,
$$N(\widetilde{H}^2\cap H'^{(L\delta^{\frac{3\tau_{\ell}}{2}})},\delta)
\geq \delta^{9K_{\ell+1}\epsilon}N(\widetilde{H},\delta)
\geq \delta^{\frac{9}{10}K_\ell\epsilon} N(\widetilde{H},\delta).$$
Note that we allow ourselves here a slight abuse of notation, denoting by $H'$ both the subgroup chunk in $U$ and the subgroup chunk in $U^2$.\\
Let $A=\widetilde{H}^2\cap H'^{(L\delta^{\frac{3\tau_\ell}{2}})}$ and $B=\widetilde{H}$.
We have
$$N(AB,\delta) \leq N(\widetilde{H}^3,\delta) \leq \delta^\epsilon N(\widetilde{H},\delta) \leq \delta^{(K_\ell-1)\epsilon}N(A,\delta),$$
so that by Rusza's Covering Lemma (see below Lemma~\ref{covering}), we find that there exists a finite set $X$ of cardinality at most $\delta^{-K_\ell\epsilon}$ such that 
$$\widetilde{H} \subset X\widetilde{H}'\cap \widetilde{H}'X,$$
where $\widetilde{H}'$ is a neighborhood of size $O(\delta)$ of the set $(\widetilde{H}^2\cap H'^{(L\delta^{\frac{3\tau_\ell}{2}})})^2$.
Provided $\delta$ is sufficiently small, we have $\widetilde{H}'\subset H'^{(\delta^{\tau_\ell})}$, where $H'$ now stands for the subgroup chunk in $U^4$.

It remains to check Condition~2.
Using that $\widetilde{H}$ is a $\delta^{-\epsilon}$-approximate subgroup, one sees that $\widetilde{H}'$ can be covered by at most $\delta^{-4\epsilon}$ translates of neighborhoods of $\widetilde{H}$ of size $O(\delta)$:
$$\widetilde{H}' \subset \bigcup_{i=1}^{\delta^{-4\epsilon}} x_i\widetilde{H}^{(O(\delta))}.$$
Let $D$ be any coset chunk in $U$ of dimension less than $\ell$.\\
For each $i$, we have,
$$N(x_i\widetilde{H}^{(O(\delta))}\cap D^{(\delta^{b\tau_\ell})},\delta)
= N(\widetilde{H}^{(O(\delta))}\cap x_i^{-1}D^{(\delta^{b\tau_\ell})},\delta)
\leq N(\widetilde{H}\cap x_i^{-1}D^{(\delta^{\frac{b\tau_\ell}{2}})},\delta),$$
and therefore, by assumption on $\ell$,
$$N(x_i\widetilde{H}^{(O(\delta))}\cap D^{(\delta^{b\tau_\ell})},\delta)
\leq \delta^{9K_\ell\epsilon} N(\widetilde{H},\delta).$$
This shows that
$$N(\widetilde{H}'\cap D^{(\delta^{b\tau_\ell})},\delta) \leq \delta^{-4\epsilon}\delta^{9K_\ell\epsilon} N(\widetilde{H},\delta) \leq \delta^{8K_\ell\epsilon} N(\widetilde{H},\delta).$$
\end{proof}

For convenience of the reader, we now give the version of Ruzsa's Covering Lemma we used in the above proof.

\begin{lemma}[Ruzsa Covering Lemma]\label{covering}
Let $G$ be a Lie group and $U$ a compact neighborhood of the identity. There exists a positive constant $L$ such that the following holds for any parameter $K\geq 1$.\\
Suppose $A$ and $B$ are subsets of $U$ such that $N(AB,\delta)\leq K N(A,\delta)$.
Then there exists a finite set $X$ of cardinality at most $LK$ such that $B$ is included in the neighborhood of size $L\delta$ of $A^{-1}AX$.
Similarly, if $N(BA,\delta)\leq K N(A,\delta)$, there exists a finite set $Y$ of cardinality at most $LK$ such that $B$ is included in the neighborhood of size $L\delta$ of $YAA^{-1}$.
\end{lemma}
\begin{proof}
Let $X=\{x_1,\dots,x_s\}$ be maximal among subsets of $B$ such that for each $i\neq j$, the translates $Ax_i$ and $Ax_j$ are away from each other by at least $2\delta$, in the sense that
$$\forall x\in Ax_i,\ \forall y\in Ax_j,\ d(x,y)> 2\delta.$$
Let $L_1>0$ such that left and right translations by elements of $U$ are $L_1$-bi-Lipschitz on $U$.
For each $i$, we have
$$N(Ax_i,\delta) \geq N(A,L_1\delta) \gg N(A,\delta).$$
The set $AB$ contains all the translates $Ax_i$, and those are $2\delta$ separated, so we find
$$N(AB,\delta) \geq \sum N(Ax_i,\delta) \gg (\card X) N(A,\delta),$$
and therefore,
$$\card X \ll K.$$
On the other hand, by maximality of $X$, if $b$ is any element of $B$, there exists an element $x_i$ is $X$ such that $Ab$ meets the neighborhood of size $2\delta$ of $Ax_i$. This shows that $d(b, A^{-1}Ax_i)\leq 2L_1\delta$ and thus,
$$B \subset A^{-1}A X^{(2L_1\delta)}.$$
\end{proof}

\subsection{Intersections of neighborhoods}

For us, an important property of neighborhoods of coset chunks is that they are stable under intersection.
Recall that if $X$ is a subset of $G$, then $X^{(\rho)}$ denotes the $\rho$-neighborhood of $X$.
The lemma we will need is as follows.

\begin{lemma}\label{cosetintersection}
Let $G$ be a real Lie group.
There exists a neighborhood $U$ of the identity in $G$ and constants $a,b>0$ such that for all $\rho>0$ sufficiently small, for any two coset chunks $C_1$ and $C_2$ in $U$, satisfying $C_1\not\subset C_2^{(\rho^a)}$, we have
$$C_1^{(\rho)}\cap C_2^{(\rho)} \subset C_0^{(\rho^b)},$$
for some coset chunk $C_0$ in $U$ with $\dim C_0 < \dim C_1$.
\end{lemma}

The proof goes into three steps. First, we study intersections of linear subspaces in a Euclidean space, then we consider intersections of subalgebras of a Lie algebra, and finally, we prove Lemma~\ref{cosetintersection}.

\begin{definition}
Given two subspaces $V_1$ and $V_2$ of a Euclidean space $E$, we define the distance from $V_1$ to $V_2$ by
$$d(V_1,V_2) = \sup\{d(v, V_2) \,;\, v\ \mbox{unit vector in}\ V_1\}.$$
\end{definition}

Note that $d$ does not define a distance on the set of subspaces of $E$, as $d(V_1,V_2)=0$ just means that $V_1$ is included in $V_2$.

\begin{lemma}\label{subspaceintersection}
Let $d$ be a positive integer. There exists a constant $c_0=c_0(d)>0$ such that if $E$ is a Euclidean space of dimension $d$, the following holds for any $\rho>0$ small enough.\\
Suppose $V_1$ and $V_2$ are two proper subspaces of $E$ such that $d(V_1,V_2)\geq\rho^{c_0}$.
Then there exists a nonnegative integer $\ell<\dim V_1$, a constant $c\geq c_0$, and an orthogonal family $(u_i)_{1\leq i\leq \ell}$ of unit vectors such that,
\begin{equation}\label{uiinv}
\forall i,\ u_i\in V_1\cap V_2^{(\rho^c)}
\end{equation}
and
$$B_E(0,1)\cap V_1\cap V_2^{(\rho^{\frac{3c}{4}})} \subset V^{(\rho^{\frac{c}{6}})},$$
where $V=\Span(u_i)_{1\leq i\leq \ell}$.
\end{lemma}
\begin{proof}
We will prove the lemma with constant $c_0=2^{-d-1}$.\\
Let $\ell\geq 0$ be maximal such that there exists an orthonormal family $(u_i)_{1\leq i\leq \ell}$ of vectors in $V_1$ such that
$$\forall i\in\{1,\dots,\ell\},\quad d(u_i,V_2)\leq\rho^{2^{-\ell}}.$$
The assumption $d(V_1,V_2)\geq\rho^{c_0}$ ensures that $\ell<\dim V_1$.
Choosing $c=2^{-\ell}$, the $u_i$'s certainly satisfy condition~(\ref{uiinv}).\\
Now let $v$ be a vector in $B_E(0,1)\cap V_1\cap V_2^{(\rho^{\frac{3c}{4}})}$.
Write $v=\lambda_1 u_1+\dots+\lambda_\ell u_\ell+ v'$, with $v'$ in $V_1\cap (u_1,\dots,u_\ell)^\perp$.
We have, provided $\rho$ is small enough,
$$d(v',V_2)=d(v-\sum\lambda_i u_i,V_2) \leq \rho^{\frac{3c}{4}}+\ell\rho^{c} \leq \rho^{\frac{2c}{3}},$$
and therefore, by maximality of $\ell$,
$$\rho^{\frac{2c}{3}} \geq \|v'\| d(\frac{v'}{\|v'\|},V_2) \geq \|v'\|\rho^{\frac{c}{2}}$$
which implies
$$d(v,V_2) = \|v'\| \leq \rho^{\frac{c}{6}}.$$
This shows that $B_E(0,1)\cap V_1\cap V_2^{(\rho^{\frac{3c}{4}})} \subset V^{(\rho^{\frac{c}{6}})}$.
\end{proof}

The next step, passing from linear subspaces to Lie subalgebras, is an application of \L ojasiewicz's inequality.

\begin{lemma}\label{algebraintersection}
Let $\g$ be a real Lie algebra endowed with a Euclidean metric.
There exist positive constants $a$ and $b$
such that for all $\rho>0$ small enough,
we have the following.
Let $\h_1$ and $\h_2$ be two Lie subalgebras of $\g$, and assume that $d(\h_1,\h_2)\geq\rho^a$. Then, there exists a Lie subalgebra $\h$ such that $\dim \h<\dim \h_1$ and
$$B_{\g}(0,1)\cap \h_1^{(\rho)} \cap \h_2^{(\rho)} \subset \h^{(\rho^b)}.$$
\end{lemma}
\begin{proof}
For each $\ell$ in $\{1,\dots,d\}$, the variety $M_\ell$ of orthogonal $\ell$-tuples of unit vectors in $\g$ is compact and real analytic. We define a real-valued function $f$ on $M_\ell$ by
$$f(u_1,\dots,u_\ell) = \sum_{1\leq i<j\leq \ell} d([u_i,u_j], \Span(u_i)_{1\leq i\leq \ell})^2,$$
where $[,]$ denotes the Lie bracket in $\g$.\\
Note that $f(u_1,\dots,u_\ell)=0$ if and only if $\Span(u_i)$ is stable under Lie brackets, i.e. if and only if the $\ell$-tuple $(u_i)$ is the basis of a Lie subalgebra of $\g$.
The function $f$ is real-analytic so that by the \L ojasiewicz inequality \cite[Théorème 2, page 62]{lojasiewicz}, there exists a constant $C$ such that for $r$ small enough,
$$f(u_1,\dots,u_\ell)\leq r \quad\Longrightarrow\quad d((u_i),Z_f)\leq r^{\frac{1}{C}},$$
where $Z_f$ is the zero set of $f$.
In other terms, if $f(u_1,\dots,u_\ell)\leq r$, then there exists a Lie subalgebra $\h$ of dimension $\ell$ such that $d(\Span(u_i),\h)\leq r^{\frac{1}{C}}$.\\
Let $c_0$ be the constant from Lemma~\ref{subspaceintersection}, and let $a=c_0$.\\
Now suppose $\h_1$ and $\h_2$ are two subalgebras as in the lemma. Choose $\ell< \dim \h_1$, $c\geq c_0$ and a orthonormal family $(u_i)_{1\leq i\leq \ell}$ as given by Lemma~\ref{subspaceintersection}.
For each $i<j$, we have $u_i$ and $u_j$ are in $\h_1$, so that $[u_i,u_j]\in\h_1$. Moreover, $u_i$ and $u_j$ are in $\h_2^{(\rho^c)}$ so, for some constant $L$ depending only on $\g$, we have $[u_i,u_j]\in\h_2^{(L\rho^c)}$. 
Thus, for $\rho>0$ small enough, $[u_i,u_j] \in \h_1\cap\h_2^{(\rho^{\frac{3c}{4}})} \subset \left(\Span(u_i)\right)^{(\rho^{\frac{c}{6}})}$,
and
$$f(u_1,\dots,u_\ell)\leq d^2\rho^{\frac{c}{6}} \leq \rho^{\frac{c}{7}}.$$
Therefore, there exists a Lie subalgebra $\h$ of dimension $\ell$ such that $d(\Span(u_i),\h)\leq\rho^{\frac{c}{7C}}$.
However, by definition of the $u_i$'s, the intersection $B_{\g}(0,1)\cap \h_1^{(\rho)}\cap\h_2^{(\rho)}$ is included in $(\Span(u_i))^{(\rho+\rho^{\frac{c}{6}})}$, so that, setting $b=\frac{c_0}{8C}$, we indeed get, provided $\rho$ is small enough,
$$\h_1^{(\rho)}\cap\h_2^{(\rho)} \subset \h^{(\rho^b)}.$$ 
\end{proof}

The above Lemma~\ref{algebraintersection} can of course be reformulated in terms of subgroup chunks, and thus allows to prove Lemma~\ref{cosetintersection}.

\begin{proof}[Proof of Lemma~\ref{cosetintersection}]
Suppose without loss of generality that the intersection $g_1H_1^{(\rho)}\cap g_2H_2^{(\rho)}$ is nonempty, and fix $g_0\in g_1H_1^{(\rho)}\cap g_2H_2^{(\rho)}$. Then we have, for some constant $L$ depending only on the compact neighborhood $U$,
$$g_0H_1^{(L\rho)}\supset g_1H_1^{(\rho)} \quad\mbox{and}\quad g_0H_2^{(L\rho)}\supset g_2H_2^{(\rho)}.$$
All we have to show is that $g_0H_1^{(L\rho)}\cap g_0H_2^{(L\rho)}$ is included in the $\rho^{b}$-neighborhood of some coset chunk.
From $g_1H_1\not\subset g_2H_2^{(\rho^a)}$, we have $g_0H_1\not\subset g_0H_2^{(\frac{\rho^a}{2})}$ whence $H_1\not\subset H_2^{(\rho^{\frac{\rho^a}{2}})}$.
Therefore, by Lemma~\ref{algebraintersection} -- adjusting slightly the values of $a$ and $b$ --, there exists a subgroup chunk $H_0$ of dimension less than $\dim H_1$ such that
$$H_1^{(L\rho)}\cap H_2^{(L\rho)} \subset H_0^{(\rho^b)},$$
and this allows us to concude that
$$g_1H_1^{(\rho)}\cap g_2H_2^{(\rho)} \subset g_0H_0^{(\rho^b)}.$$
\end{proof}

\section{The set $\Xi$ of troublemakers}
\label{section:troublemakers}

Let $G$ be a Lie group and $U$ a compact neighborhood of the identity.
If $A\subset U$ is a $(\sigma,\epsilon)$-set at scale $\delta$ in a Lie group $G$, we associate to it the set $\Xi$ of \emph{troublemakers} for $A$, defined as
\begin{equation}\label{tm}
\Xi = \left\{ \xi\in U \,\left|\, \exists \Omega\subset A\times A \ \mbox{with}
\begin{array}{l}
N(\Omega,\delta)\geq\delta^{\epsilon}N(A,\delta)^2 \\
\mbox{and}\ N(\pi_\xi(\Omega),\delta)\leq\delta^{-\epsilon}N(A,\delta)
\end{array}
\right.
\right\},
\end{equation}
where $\pi_\xi:(x,y)\mapsto x\xi y$.
Roughly speaking, an element $\xi$ is a troublemaker for the set $A$ if there exist large portions $A'$ and $B'$ of $A$ such that the product set $A'\xi B'$ is not much larger than $A$.

\begin{example}
Suppose $A=H\cap U$ for some closed subgroup of the Lie group $G$. If $\xi\in U$ is any element of the normalizer $N_G(H)$ of $H$, we have $A\xi A \subset H\xi\cap U^3$, which has roughly the same size as $A$. So $\Xi$ contains $N_G(H)\cap U$.
\end{example}

\subsection{Controlling the troublemakers}

The purpose of this subsection is to show that if $A$ is a $(\sigma,\epsilon)$-set in a simple Lie group $G$, then the set of troublemakers for $A$ is included in a small number of neighborhoods of cosets of proper closed subgroups of $G$.
Our aim is the following.

\begin{proposition}\label{troublemakers}
Let $G$ be a simple Lie group. There exists a neighborhood $U$ of the identity in $G$ such that, given $\sigma\in (0,\dim G)$, there exist constants $\eta=\eta(\sigma)>0$ and $\epsilon_1=\epsilon_1(\sigma)>0$ such that the following holds for any $\epsilon\in(0,\epsilon_1)$ and any $\delta>0$ small enough.\\
If $A\subset U$ is a $(\sigma,\epsilon)$-set at scale $\delta$, then the set $\Xi$ of troublemakers for $A$, defined as in (\ref{tm}), is included in a union of at most $\delta^{-O(\epsilon)}$ neighborhoods of size $\delta^{\eta}$ of coset chunks in $U$.\\
Moreover, $\eta$ and $\epsilon_1$ remain bounded away from $0$ when $\sigma$ varies in a compact subset of $(0,\dim G)$.
\end{proposition}

First, we recall the following proposition on \enquote{almost stabilizers} of subspaces in the adjoint representation of a simple Lie group \cite[Proposition~2.7]{saxceproducttheorem}.

\begin{proposition}\label{rhoaway}
Let $G$ be a simple Lie group with trivial center.
There exists a neighborhood $U$ of the identity in $G$, and a constant $c>0$ such that for all $\rho>0$ small enough, the following holds.\\
For each proper subspace $V<\g$, there exists a proper closed connected subgroup $S_V$ such that for all $\xi$ in $U$,
$$d((Ad\xi)V,V)\leq\rho \quad\Longrightarrow\quad d(\xi,S_V)\leq \rho^c.$$
\end{proposition}

That proposition has the following corollary.

\begin{corollary}\label{almoststab}
Let $G$ be a simple Lie group.
There exists a neighborhood $U$ of the identity and a constant $c>0$ such that the following holds for any $\rho>0$ sufficiently small.\\
For any two proper subgroup chunks $H$ and $R$ of same dimension in $U$, there exists a coset chunk $C$ in $U$ such that, for all $\xi$ in $U$,
$$\xi H\xi^{-1} \subset R^{(\rho)} \quad\Longrightarrow\quad \xi\in C^{(\rho^c)}.$$
\end{corollary}
\begin{proof}
Since the statement only involves a neighborhood of the identity, it is enough to prove it in the case the group $G$ has trivial center.
Choose a neighborhood $U$ of the identity and a constant $c>0$ such that Proposition~\ref{rhoaway} holds for $UU^{-1}$ and $2c$.\\
Given two subgroup chunks $H$ and $R$ having the same dimension, assume that for some $\xi_0$ in $U$, $\xi_0 H \xi_0^{-1}$ is included in $R^{(\rho)}$.
If $\xi$ is another element satisfying $\xi H\xi^{-1}\subset R^{(\rho)}$, we have, for some constant $L$ depending only on $U$,
$$ (\xi_0^{-1}\xi)H(\xi\xi_0^{-1})^{-1} \subset H^{(L\rho)}.$$
and this implies, if $\h$ denotes the Lie algebra of $H$,
$$ d((\Ad \xi_0^{-1}\xi)\h,\h) \leq L\rho.$$
From Proposition~\ref{rhoaway}, it follows that there exists a closed subgroup $S$ such that for $\rho$ small enough,
$$\xi H\xi^{-1} \subset R^{(\rho)} \quad\Longrightarrow\quad \xi\in \xi_0S^{(L\rho^{2c})}\subset\ \xi_0 S^{(\rho^c)}.$$
This proves the lemma.
\end{proof}

The proof of Proposition~\ref{troublemakers} is based on the following lemma, which is an application of the inclusion-exclusion principle.

\begin{lemma}\label{fewcosets}
Let $G$ be a Lie group, and fix a neighborhood $U$ of the identity and constants $a$ and $b$ as given by Lemma~\ref{cosetintersection}.\\
Let $\sigma\in(0,d)$ and fix an integer $\ell\in\{1,\dots,d-1\}$.
Let $A$ be any subset of $U$, and $\delta>0$ some small scale.
Given two parameters $K$ and $\tau>0$, we are interested in subsets $\widetilde{C}$ of $U$ such that:
\begin{enumerate}
\item $\delta^{K\epsilon}N(A,\delta) \leq N(\widetilde{C}\cap A,\delta)$
\item We have $\widetilde{C}\subset C^{(\delta^\tau)}$, for some coset chunk $C$ of dimension $\ell$.
\item If $D$ is any coset chunk such that $\dim D<\ell$, then\\
$N(\widetilde{C}\cap D^{(\delta^{b\tau})},\delta) \leq \delta^{4K\epsilon}N(A,\delta)$.
\end{enumerate}
Let $\mathcal{C}$ be the union of all such sets $\widetilde{C}$.
Then, there exists a family $(C_i)_{1\leq i\leq 2\delta^{-K\epsilon}}$ of coset chunks of dimension $\ell$ such that
$$\mathcal{C} \subset \bigcup_{i=1}^{2\delta^{-K\epsilon}} C_i^{(\delta^{a\tau})}.$$
\end{lemma}
\begin{proof}
Choose successively sets $\widetilde{C}_i$, $i\geq 1$ satisfying all requirements of the lemma -- in particular $\widetilde{C}_i\subset C_i^{(\delta^\tau)}$ for some $\ell$-dimensional coset chunk $C_i$ -- and such that for each $i$,
$$C_{i+1}^{(\delta^\tau)} \not\subset \bigcup_{k=1}^i C_k^{(\delta^{a\tau})}.$$
Clearly, this procedure must stop, and when it does, we obtain a finite family $(C_i)_{1\leq i\leq N}$ of coset chunks such that
$$\mathcal{C} \subset \bigcup_{i=1}^{N} C_i^{(\delta^{a\tau})}.$$
It remains to check that $N\leq 2\delta^{-K\epsilon}$.
For that, first note that for all $1\leq i<j\leq N$, by Lemma~\ref{cosetintersection}, there exists a coset chunk $D$ with $\dim D<\ell$
such that $C_i^{(\delta^{\tau})}\cap C_j^{(\delta^{\tau})} \subset D^{(\delta^{b\tau})}$.
In particular,
$$N(\widetilde{C}_i\cap \widetilde{C}_j,\delta)  \leq N(\widetilde{C}_j\cap C_i^{(\delta^{\tau})}\cap C_j^{(\delta^{\tau})},\delta)
 \leq N(\widetilde{C}_j\cap D^{(\delta^{b\tau})},\delta),$$
whence, using the third assumption on $\widetilde{C_j}$,
$$N(\widetilde{C}_i\cap \widetilde{C}_j,\delta) \leq \delta^{4K\epsilon}N(A,\delta).$$
Now, as $A$ certainly contains $\bigcup_{i=1}^N A\cap \widetilde{C}_i$, we find
\begin{IEEEeqnarray*}{rCl}
N(A,\delta) & \geq & N(A\cap \widetilde{C}_1,\delta) + (N(A\cap \widetilde{C}_2,\delta)-N(A\cap \widetilde{C}_1\cap \widetilde{C}_2,\delta))\\
&& + (N(A\cap \widetilde{C}_3,\delta)-N(A\cap \widetilde{C}_1\cap \widetilde{C}_3,\delta)-N(A\cap \widetilde{C}_2\cap \widetilde{C}_3,\delta)) + \dots\\
& \geq & \delta^{K\epsilon}N(A,\delta)\left[1+(1-\delta^{3K\epsilon})+(1-2\delta^{3K\epsilon})+\dots\right],
\end{IEEEeqnarray*}
keeping only the first $\min(N,\delta^{-3K\epsilon})$ terms in the sum.
The terms on the right-hand side of the above inequality are non-negative and form an arithmetic progression, so that we get the lower bound
$$N(A,\delta) \geq \delta^{K\epsilon}\frac{1}{2}\min(N,\delta^{-3K\epsilon})N(A,\delta).$$
This forces $\min(N,\delta^{-3K\epsilon})=N$ and in turn,
$$N\leq 2\delta^{-K\epsilon}.$$
\end{proof}

\begin{proof}[Proof of Proposition~\ref{troublemakers}]
\underline{A}. Choose a symmetric neighborhood $U$ such that both Lemma~\ref{niceh} and Lemma~\ref{dfirst} hold in the neighborhood $U^4$.
Let $\xi$ be an element of $\Xi$.
From the non-commutative version of the Balog-Szemerédi-Gowers Lemma, due to Tao \cite[Theorem~6.10]{taoestimates}, there exists a constant $K\geq 2$
such that there exists a $\delta^{-K\epsilon}$-approximate subgroup $\widetilde{H}$ and elements $x,y$ in $G$
such that
$$\delta^{K\epsilon}N(A,\delta) \leq N(x\widetilde{H}\cap A,\delta) \leq N(\widetilde{H},\delta) \leq \delta^{-K\epsilon}N(A,\delta)$$
and
$$\delta^{K\epsilon}N(A,\delta) \leq N(\xi^{-1}\widetilde{H}y\cap A,\delta) \leq N(\widetilde{H},\delta) \leq \delta^{-K\epsilon}N(A,\delta).$$

\noindent\underline{B}. First, we claim that $\widetilde{H}$ is a $(\sigma,3K\epsilon)$-set at scale $\delta$.\\
Indeed, suppose for a contradiction that for some ball $B_\rho$ of radius $\rho\geq\delta$, we have
$$N(\widetilde{H}\cap B_\rho,\delta) > \delta^{-3K\epsilon}\rho^\sigma N(\widetilde{H},\delta).$$
Then,
$$N((\widetilde{H}\cap B_\rho)\widetilde{H},\delta)
\leq N(\widetilde{H}^2,\delta)
\leq \delta^{-K\epsilon} N(\widetilde{H},\delta),$$
so that
$$N((\widetilde{H}\cap B_\rho)\widetilde{H},\delta)
\leq \delta^{2K\epsilon}\rho^{-\sigma} N(\widetilde{H}\cap B_\rho,\delta).$$
Applying the Covering Lemma~\ref{covering} to the sets $\widetilde{H}\cap B_\rho$ and $\widetilde{H}$, we find that for some constant $L$ depending on $U$ only, there is a set $W$ of cardinality at most $L\delta^{2K\epsilon}\rho^{-\sigma}$ such that,
$$\widetilde{H} \subset W\cdot(\widetilde{H}^2\cap B(1,L\rho)) \subset \bigcup_{w\in W}B(w,L\rho).$$
Recalling that $N(x\widetilde{H}\cap A,\delta)\geq\delta^{K\epsilon}N(A,\delta)$, we see that for some $w$ in $W$, we have
$$N(A\cap B(xw,L\rho),\delta) \geq \frac{1}{\card W} \delta^{K\epsilon}N(A,\delta) \geq \frac{1}{L}\delta^{-K\epsilon}\rho^{\sigma}N(A,\delta),$$
contradicting the fact that $A$ is $(\sigma,\epsilon)$-set at scale $\delta$, since for $\delta$ small enough,
$$\frac{1}{L}\delta^{-K\epsilon}\rho^{\sigma} > (L\rho)^{\sigma}\delta^{-\epsilon}.$$\\

\noindent\underline{C}. Now, let $\epsilon_0$, $K_\ell$ and $\tau_\ell$, $1\leq \ell\leq d-1$, be as in Lemma~\ref{niceh}.
Provided $\epsilon<\epsilon_1:=\frac{\epsilon_0}{2K}$, Lemma~\ref{niceh} shows that there is an integer $\ell\in\{1,\dots,d-1\}$ and a subgroup chunk $H'$
of dimension $\ell$ such that, for some set $\widetilde{H}'\subset  H'^{(\tau_\ell)}$,
\begin{enumerate}
\item There is a finite set $X$ of cardinality at most $\delta^{-3K_\ell K\epsilon}$ such that $\widetilde{H}\subset X\widetilde{H}'\cap \widetilde{H}'X$.
\item If $D$ is any coset chunk in $U$ such that $\dim D<\ell$, then\\
$N(\widetilde{H}'\cap D^{(\delta^{b\tau_\ell})},\delta) \leq \delta^{24K_\ell K\epsilon}N(\widetilde{H},\delta)
\leq \delta^{23K_\ell K\epsilon}N(A,\delta)$.
\end{enumerate}
We have $x\widetilde{H}\subset xX\widetilde{H}'$, and recalling that $N(x\widetilde{H}\cap A,\delta)\geq\delta^{K\epsilon}N(A,\delta)$, we find that for some $x'$ in $xX$,
$$N(x'\widetilde{H}'\cap A,\delta) \geq \frac{1}{\card X}\delta^{K\epsilon}N(A,\delta) \geq \delta^{4K_\ell K\epsilon}N(A,\delta).$$
Similarly, there exists $y'$ such that
$$N(\xi^{-1}\widetilde{H}'y'\cap A,\delta) \geq \delta^{4K_\ell K\epsilon}N(A,\delta).$$
Denote by $\mathcal{C}$ the union of all subsets $\widetilde{C}$ of $U^4$ satisfying the conditions of Lemma~\ref{fewcosets}, with constants $K'=4K_\ell K$ and $\tau=\tau_\ell$.
By Lemma~\ref{fewcosets}, there is a family of coset chunks $(C_i)_{1\leq i\leq 2\delta^{-4K_\ell K\epsilon}}$ such that
$$\mathcal{C} \subset \bigcup_{i=1}^{2\delta^{-4K_\ell K\epsilon}} C_i^{(\delta^{a\tau})}.$$

\noindent\underline{D}. As $x'\widetilde{H}'$ and $\xi^{-1}\widetilde{H}'y'$ both satisfy the conditions of Lemma~\ref{fewcosets}, there must exist indices $i$ and $j$ such that
\begin{equation}\label{iandj}
x'\widetilde{H}' \subset C_i^{(\delta^{a\tau})} \quad\mbox{and}\quad \xi^{-1}\widetilde{H}'y' \subset C_j^{(\delta^{a\tau})}.
\end{equation}
Denote by $H_i$ the left-direction of $C_i$, i.e. the subgroup chunk such that there exists $x_i$ such that $C_i=x_iH_i$, and by $R_j$ the right-direction of $C_j$.
From (\ref{iandj}), we get
$$\xi^{-1}H_i\xi \subset R_j^{(\delta^{a\tau})},$$
and therefore, by Corollary~\ref{almoststab}, for some small $c>0$ depending only on $U$,
$$\xi \in C_{i,j}^{(\delta^{ca\tau})},$$
where $C_{i,j}$ is a left-coset of some proper maximal closed subgroup in $G$.\\
Letting $\eta=ca\tau_{d-1}=\min_{1\leq \ell\leq d-1}ca\tau_\ell$ and considering all (at most $\delta^{-O(\epsilon)}$) cosets $C_{i,j}$ arising for some dimension $\ell\in\{1,\dots,d-1\}$, this proves the proposition.
\end{proof}

\subsection{Escaping from the troublemakers}

We start by a lemma that will allow us to escape from hyperplanes in the adjoint representation.

\begin{lemma}\label{parameters}
Let $G$ be a connected simple real Lie group of dimension $d$. There exist integers $k$ and $s$ such that for any set $A$ containing $1$ and generating a dense subgroup of $G$, there exists a finite family $(a_i)_{1\leq i\leq k}$ of elements of the product set $A^s$
such that for any nonzero vector $v\in \g$ and any hyperplane $V<\g$, there exists an index $i$ for which
$$(\Ad a_i)v \not\in V.$$
\end{lemma}
\begin{proof}
Let $A$ be a topologically generating set of $G$, and fix $(a,b)$ a topologically generating pair of elements of $A$ (for the existence of such a pair, see for example \cite{breuillardgelander_densefreesubgroups}). By induction on $\ell<d$ we will show that for any nonzero vector $v$ in $\g$ and any subspace $V$ of dimension $\ell$, there exist elements $a_1,\dots,a_\ell$ in $\{1,a,b\}$ such that
$$a_\ell\dots a_1\cdot v \not\in V.$$
\underline{$\ell=1$}\\
If $v\not\in V$, just take $a_1=1$. Otherwise, $V=\R v$. The stabilizer $\Stab V$ of the line $V$ is a proper closed subgroup of $G$ and therefore we must have $a\not\in\Stab V$ or $b\not\in\Stab V$. This shows that $av\not\in V$ or $bv\not\in V$.\\
\underline{$\ell\rightarrow\ell+1$}\\
Suppose we know the result for subspaces of dimension at most $\ell$, and let $V$ be a proper subspace of $\g$ of dimension $\ell+1<d$.
Let $V_1=V\cap a^{-1}V\cap b^{-1}V$. Either $a$ or $b$ is out of the proper closed subgroup $\Stab V$, so we must have $\dim V_1\leq\ell$. By the induction hypothesis, there exist elements $a_1,\dots,a_\ell$ in $\{1,a,b\}$ such that $a_\ell\dots a_1\cdot v \not\in V_1$. By definition of $V_1$ this shows that for some $a_{\ell+1}$ in $\{1,a,b\}$, we must have
$$a_{\ell+1}a_\ell\dots a_1\cdot v \not\in V.$$
This proves the lemma, with constants $s=d-1$ and $k=2^d-1$ (the number of words of length at most $d-1$ in $a$ and $b$).
\end{proof}

The goal of this subsection is to show that one can always escape from the set $\Xi$ of troublemakers of a $(\sigma,\epsilon)$-set $A$, in the following precise sense.

\begin{proposition}\label{outofxi}
Let $G$ be a simple Lie group. There exists a neighborhood $U$ of the identity such that, given $\sigma\in(0,d)$, there exists $\epsilon=\epsilon(\sigma)>0$ such that the following holds.\\
Suppose $\{a_i\}_{1\leq i\leq k}$ is a finite set of elements satisfying the conclusion of Lemma~\ref{parameters}, let $n=d^d$ and consider the set $\Pi$ of all projections $\pi:G^{\times n} \rightarrow G$ of the form
$$\pi(x_1,\dots,x_n) = x_{i_1} a_{j_1} x_{i_2} a_{j_2} \dots a_{j_{m-1}}x_{i_m},$$
where $\{i_1<i_2<\dots<i_m\}\subset \{1,\dots,n\}$ and $(j_1,\dots,j_{m-1})\in \{1,\dots,k\}^{m-1}$.\\
Then, for all $\delta>0$ small enough, if $A\subset U$ is a $(\sigma,\epsilon)$-set at scale $\delta$, and if $\Omega$ is any subset of the Cartesian product set $A^{\times n}$ such that $N(\Omega,\delta) \geq \delta^{\epsilon}N(A,\delta)^n$, then there exists some $\pi$ in $\Pi$ such that
$$\pi(\Omega) \not\subset \Xi,$$
where $\Xi$ is the set of troublemakers for $A$, as defined in (\ref{tm}).
\end{proposition}

The proof of Proposition~\ref{outofxi} will be based on a repeated application of the following lemma.

\begin{lemma}\label{dfirst}
Let $G$ be a simple Lie group. There exist a neighborhood $U$ of the identity and a constant $b>0$ such that given $c_0>0$, for all $\rho>0$ small enough (in terms of $c_0$), the following holds.\\
Let $A$ be a subset of $U$, and $\Omega$ a subset of the Cartesian product set $A^{\times n}$ ($n\geq d$).
Assume that there exist coset chunks $C_i$, $1\leq i\leq d$ in $U$ such that
$$\Omega\subset C_1^{(\rho)}\times\dots\times C_d^{(\rho)}\times U\times\dots\times U.$$
For each $i$, write $C_i=g_iH_i$, for some subgroup chunk $H_i$ and some element $g_i$ in $U$, and suppose that there exist elements $a_i$, $1\leq i\leq d-1$ in $U$ and unit vectors $v_i\in\h_i$, such that, denoting $t_i=a_ig_{i+1}a_{i+1}g_{i+2}\dots a_{d-1}g_d$, we have,\\
for each $i$ in $\{1,\dots, d-1\}$,
\begin{equation}\label{tis}
d((\Ad t_i^{-1})v_i, \bigoplus_{j=i+1}^d\R (\Ad t_j^{-1})v_j) \geq c_0.
\end{equation}
Finally, let $\pi$ be the map $G^{\times n}\rightarrow G$ defined by
$$\pi(x_1,\dots,x_n) = x_1a_1\dots x_{d-1}a_{d-1}x_d$$
and assume that for some proper coset chunk $C$ in $U$,
$$\pi(\Omega) \subset C^{(\rho)}.$$
Then there exists an index $i_0$ in $\{1,\dots,d\}$ and elements $u_i\in U$, $1\leq i\leq d$, $i\neq i_0$, such that the set
$$\Omega' = \{(x_{i_0},x_{d+1},\dots,x_n)\in A^{\times n-d+1} \,|\, (u_1,\dots,u_{i_0-1},x_{i_0},u_{i_0+1},\dots,u_d,x_{d+1},\dots,x_n)\in\Omega\}$$
satisfies
\begin{enumerate}
\item $N(\Omega',\delta) \geq \frac{N(\Omega,\delta)}{N(A,\delta)^{d-1}}$
\item $\Omega'\subset C'^{(\rho^b)}\times U\times\dots\times U$, for some coset chunk $C'$ in $U$ satisfying $\dim C'< \max_{1\leq i\leq d} \dim C_i$.
\end{enumerate}
\end{lemma}
\begin{proof}
Choose $U$ such that Lemma~\ref{cosetintersection} holds.\\
Let $c_1>0$ be as in Lemma~\ref{separatedbasis} below.
Write $C=gH$ for some $g$ in $U$ and some subgroup chunk $H$ with Lie algebra $\h$.
By Lemma~\ref{separatedbasis} applied to the family $(u_i)=((\Ad t_i)^{-1}v_i)$, there exists an index $i$ in $\{1,\dots,d\}$ for which 
$$d((\Ad t_{i}^{-1})v_{i},\h) \geq c_1.$$
Fix a large constant $L\geq 2$,
and let $i_0\in\{1,\dots,d\}$ be maximal such that
$$d((\Ad t_{i_0}^{-1})\h_{i_0},\h) \geq \frac{c_1}{L^{i_0}}.$$
Then choose any elements $u_i$, $i\neq i_0$, $1\leq i\leq d$ so that the set $\Omega'$ defined in the lemma satisfies
$$N(\Omega',\delta) \geq \frac{N(\Omega,\delta)}{N(A,\delta)^{d-1}}.$$

Now let $(x_{i_0},x_{d+1},\dots,x_n)$ be an element of $\Omega'$. We want to show that $x_{i_0}$ stays in a neighborhood of size $\rho^b$ of a coset chunk $C'$ satisfying $\dim C' < \dim C_{i_0}$.
This will follow from
$$x_{i_0}\in g_{i_0}H_{i_0}^{(\rho)} \quad\mbox{and}\quad u_1a_1\dots a_{i_0-1}x_{i_0}a_{i_0}\dots u_d \in gH^{(\rho)}.$$
For simplicity, denote $t=t_{i_0}$ and $s=t_{i_0}^{-1}a_{i_0}\dots u_d$, so that the above can be rewritten, for some $u$ in $U$,
$$x_{i_0}\in g_{i_0}H_{i_0}^{(\rho)} \quad\mbox{and}\quad x_{i_0} \in utsH^{(\rho)}s^{-1}t^{-1}.$$
To conclude, we want to apply Lemma~\ref{cosetintersection}, but for that, we need to check that the two coset chunks above are away from one another.
For each $i$, write $u_i= g_i h_i$, for some $h_i$ is the subgroup chunk $H_i$, so that
$$s = t_{i_0+1}^{-1}h_{i_0+1}t_{i_0+1}\dots t_{d-1}^{-1}h_{d-1}t_{d-1}h_d.$$
By maximality of $i_0$, we have, for each $i>i_0$, $d((\Ad t_i)^{-1}\h_i,\h)\leq \frac{c_1}{L^i}$,\\
so that in particular, for some constant $L_0$ depending only on $U$,
$$\forall i>i_0,\quad d(t_i^{-1}h_it_i,H)\leq\frac{L_0c_1}{L^i}.$$
This shows that $s$ is quite close to $H$:
$$
d(s,H) \leq \sum_{i>i_0} \frac{L_0c_1}{L^i}
= \frac{2L_0c_1}{L^{i_0+1}},
$$
which implies, for some constant $L_1$ depending only on $U$, 
$$d(sHs^{-1},H) \leq \frac{L_1c_1}{L^{i_0+1}}.$$
On the other hand, by our choice of $i_0$, we have, for some constant $L_2$ depending on $U$ only,
$$d(t^{-1}H_{i_0}t, H) \geq \frac{c_1}{L_2L^{i_0}}.$$
Now, for some constant $L_3$ depending only on $U$,
\begin{align*}
d(H_{i_0},tsHs^{-1}t^{-1}) & \geq \frac{1}{L_3}d(t^{-1}H_{i_0}t,sHs^{-1})\\
& \geq \frac{1}{L_3}[d(t^{-1}H_{i_0}t,H)-d(H,sHs^{-1})],
\end{align*}
so that
\begin{align*}
d(H_{i_0},tsHs^{-1}t^{-1}) &  \geq \frac{1}{L_3}[\frac{c_1}{L_2L^{i_0}}-\frac{L_1c_1}{L^{i_0+1}}]\\
& \geq \frac{c_1}{L^{i_0+1}}
\end{align*}
provided $L$ has been chosen larger than $L_2(L_3+L_1)$.

\bigskip
To conclude, let $a$ and $b$ be the constants from Lemma~\ref{cosetintersection}.
The above inequality ensures that for $\rho>0$ sufficiently small,
$$d(H_{i_0},tsHs^{-1}t^{-1}) \geq \rho^a,$$
so that there exists a coset chunk $C'$ with $\dim C'<\dim C_{i_0}$ and
$$g_{i_0}H_{i_0}^{(\rho)}\cap utsH^{(\rho)}s^{-1}t^{-1} \subset C'^{(\rho^b)}.$$
Thus,
$$\Omega' \subset C'^{(\rho^b)}\times A\times\dots\times A,$$
and the lemma is proven.
\end{proof}

At the beginning of the above proof, we made use of the following easy lemma.

\begin{lemma}\label{separatedbasis}
Let $E$ be a Euclidean vector space of dimension $d$. Given $c_0>0$, there exists a constant $c_1>0$ such that the following holds.\\
Suppose $(u_i)_{1\leq i\leq d}$ is a family of unit vectors of $E$ such that
\begin{equation}\label{sep}
\forall i\in\{1,\dots,d-1\},\quad d(u_{i+1},\Span(u_j)_{1\leq j\leq i})\geq c_0.
\end{equation}
Then, for all proper linear subspace $W<E$, there exists an index $i$ for which
$$d(u_i,W) \geq c_1.$$
\end{lemma}
\begin{proof}
Given a $d$-tuple $(u_i)$ of elements of $\g$ and a hyperplane $W<\g$, we define
$$\varphi(u_1,\dots,u_d,W) = \max_{1\leq i\leq d} d(u_i,W).$$
The map $\varphi$ is continuous, and strictly positive whenever $(u_i)$ is a basis for $\g$. The set $E_{c_0}$ of $d$-tuples $(u_i)$ of unit vectors satisfying~(\ref{sep}) is compact, and so is the Grassmannian variety $\mathcal{G}$ of hyperplanes of $\g$.
This proves the lemma, with constant $c_1$ equal to the minimal value of $\varphi$ on the compact set $E_{c_0}\times \mathcal{G}$.
\end{proof}

For simplicity, if $\Omega$ is a subset of the Cartesian product set $G^{\times n}$, we will say that a set $\Omega'$ \emph{comes from} $\Omega$ if there exist indices $i_1<i_2<\dots<i_r$ in $\{1,\dots,n\}$ and elements $x_{i_1},\dots,x_{i_r}$ such that
$$\Omega' \subset \left\{\left.(x_i)_{\substack{1\leq i\leq n\\ i\not\in\{i_1,\dots,i_r\}}}
\,\right|\, (x_i)_{1\leq i\leq n} \in \Omega\right\}.$$
We now turn to the proof of Proposition~\ref{outofxi}.

\begin{proof}[Proof of Proposition~\ref{outofxi}]
The neighborhood $U$ is chosen such that Lemma~\ref{dfirst} holds.\\
Recall from Proposition~\ref{troublemakers} that there exists a finite family $(C_i)_{1\leq i\leq \delta^{-O(\epsilon)}}$ of coset chunks in $U$ such that
\begin{equation}\label{xismall}
\Xi \subset \bigcup_{1\leq i\leq\delta^{-O(\epsilon)}} C_i^{(\delta^\eta)}.
\end{equation}
Let $\Omega$ be a subset of the Cartesian product $A^{\times n}$ such that $N(\Omega,\delta)\geq \delta^{\epsilon}N(A,\delta)^n$, and assume for a contradiction that for all $\pi$ in $\Pi$,
\begin{equation}\label{inxi}
\pi(\Omega) \subset \Xi.
\end{equation}
For $s\in\{0,\dots,d-1\}$, we let $n_s= d^{d-s}$ and $\eta_s=b^s\eta$.
To reach a contradiction, we apply Lemma~\ref{dfirst} inductively. We decompose the reasoning into (at most) $d$ steps.\\
\underline{Step 0}\\
Just using the inclusion (\ref{inxi}) for all projections on the coordinates, and recalling that $\Xi$ is controlled by (\ref{xismall}), we see that $\Omega$ is included in a union of at most $\delta^{-O(\epsilon)}$ sets of the form $C_1^{(\delta^\eta)}\times\dots\times C_n^{(\delta^\eta)}$, where the $C_i$'s are proper coset chunks in $U$. By the pigeonhole principle, there must exist coset chunks $C_{01},\dots,C_{0n}$ of dimension at most $d-1$ and a set $\Omega_0\subset\Omega$ such that
\begin{enumerate}
\item $N(\Omega_0,\delta) \geq \delta^{O(\epsilon)}N(\Omega,\delta)$
\item $\Omega_0 \subset C_{01}^{(\delta^\eta)}\times\dots\times C_{0n}^{(\delta^\eta)}$
\end{enumerate}
\underline{Step $s+1$, $s\geq 0$}\\
Suppose we have constructed a set $\Omega_s\subset A^{\times n_s}$ coming from $\Omega$, and coset chunks $C_{s 1},\dots,C_{s n_s}$ of dimension at most $d-1-s$ and at least $1$ such that
\begin{enumerate}
\item $N(\Omega_s,\delta) \geq \delta^{O(\epsilon)}N(A,\delta)^{n_s}$
\item $\Omega_s \subset C_{s 1}^{(\delta^{\eta_s})}\times\dots\times C_{s n_s}^{(\delta^{\eta_s})}$
\end{enumerate}
For each $i\in\{1,\dots,n_s\}$, write $C_{s i}=g_iH_i$ for some subgroup chunk $H_i$ and some element $g_i$ in $U$.
By assumption on the family $\{a_i\}$, there exists a constant $c_0>0$ such that for all unit vector $v\in\g$ and all hyperplane $W<\g$, there exists an element $a_i$ such that $d((\Ad a_i)v, W)\geq c_0$.
This allows us to choose $a_1,\dots,a_d$ among the $a_i$'s so that condition~(\ref{tis}) of Lemma~\ref{dfirst} is satisfied (for some constant $c_0$ depending only on the set of parameters $\{a_i\}$). Denote by $\pi$ the associated projection.
From the inclusions $\pi(\Omega_s)\subset\Xi$ and (\ref{xismall}), we see by the pigeonhole principle that there exists a coset chunk $C=gH$ in $U$ and a subset $\Omega_s'\subset\Omega_s$ such that $N(\Omega_s',\delta)\geq \delta^{O(\epsilon)}N(\Omega_s,\delta)$ and $\pi(\Omega_s')\subset C^{(\delta^\eta)}$.\\
We now apply Lemma~\ref{dfirst} to $\Omega_s'$, at scale $\rho=\delta^{\eta_s}$, and get a set $\Omega_{s  1}\subset A^{\times n_s-d+1}$ coming from $\Omega_s$ and a coset chunk $C_{(s+1)1}$ in $U$ of dimension at most $d-2-s$ such that
\begin{enumerate}
\item $N(\Omega_{s 1},\delta) \geq \delta^{O(\epsilon)}N(A,\delta)^{n_s-d+1}$
\item $\Omega_{s 1} \subset C_{(s+1)1}^{(\delta^{\eta_{s+1}})}\times C_{s (d+1)}^{(\delta^{\eta_s})}\dots\times C_{s n_s}^{(\delta^{\eta_s})}$
\end{enumerate}
Repeating this argument with the next $d$ coordinates, and then again with the $d$ following, etc., we finally get a set $\Omega_{s+1}$ coming from $\Omega_s$ and included in the Cartesian product $A^{\times n_{s+1}}$, and coset chunks $C_{(s+1)1},\dots,C_{(s+1)n_{s+1}}$ of dimension at most $d-2-s$ in $U$ such that
\begin{enumerate}
\item $N(\Omega_{s+1},\delta) \geq \delta^{O(\epsilon)}N(A,\delta)^{n_{s+1}}$
\item $\Omega_{s+1} \subset C_{(s+1)1}^{(\delta^{\eta_{s+1}})}\times\dots\times C_{(s+1)n}^{(\delta^{\eta_{s+1}})}$
\end{enumerate}
As the dimensions of the coset chunks $C_{s i}$ are bounded above by $d-s-1$, we must obtain, for some $s\leq d-1$ and some $i\in\{1,\dots,n_s\}$ that $\dim C_{s i}=0$.
In other terms, the set $C_{s i}$ is reduced to a point, so that the projection $S$ of $\Omega_s$ on its $i$-th coordinate is included in a ball of radius $\delta^{\eta_s}$.
By construction, $S$ is included in $A$, so that recalling that $A$ is a $(\sigma,\epsilon)$-set at scale $\delta$, we find
$$N(S,\delta) \leq \delta^{\sigma\eta_s-\epsilon}N(A,\delta).$$
However, from the lower bound $\delta^{O(\epsilon)}N(A,\delta)^{n_s}$ on the cardinality of $\Omega_s$, it is readily seen that
$$N(S,\delta) \geq \delta^{O(\epsilon)}N(A,\delta),$$
which yields the desired contradiction, provided $\epsilon$ has been chosen small enough.
\end{proof}

Let $n=d^d$, and $\Pi$ be the set of projections $G^{\times n}\rightarrow G$ as defined in Proposition~\ref{outofxi}. Let $N$ denote the cardinality of $\Pi$, and consider the map
\begin{equation}\label{ww}
\begin{array}{lccc}
w: & G^{\times n+N+1} & \rightarrow & G\\
& (x_1,\dots,x_n,y_0,\dots,y_N) & \mapsto & y_0\pi_1(x_1,\dots,x_n)y_1\dots y_{N-1}\pi_{N}(x_1,\dots,x_n)y_N
\end{array}
\end{equation}

Proposition~\ref{outofxi} has the following corollary on expansion of $(\sigma,\epsilon)$-sets in the simple Lie group $G$.

\begin{corollary}\label{finalcombinatorial}
Let $G$ be a simple Lie group.
There exists a neighborhood $U$ of the identity such that, given $\sigma\in(0,d)$, there exists $\epsilon=\epsilon(\sigma)>0$ such that the following holds.\\
Suppose $(a_i)$ is a family of elements of $U$ satisfying the conclusion of Lemma~\ref{parameters}, and let $w:G^{\times n+N+1}\rightarrow G$ be the associated map, as defined above.\\
For all $\delta>0$ sufficiently small, if $A\subset U$ is a $(\sigma,\epsilon)$-set at scale $\delta$ and $\Omega$ is a subset of the Cartesian product $A^{\times n+N+1}$ satisfying $N(\Omega,\delta)\geq \delta^\epsilon N(A,\delta)^{n+N+1}$, then
$$N(w(\Omega),\delta) \geq \delta^{-\epsilon}N(A,\delta).$$
\end{corollary}
\begin{proof}
For a $n$-tuple $x=(x_1,\dots,x_n)$ of elements of $U$, we denote
$$\Omega_x = 
\{y=(y_0,\dots,y_N) \,|\, (x_1,\dots,x_n,y_0,\dots,y_N)\in\Omega\}.$$
Let
$$\Omega'=\{(x,y)\in\Omega \,|\, N(\Omega_x,\delta) \geq \frac{\delta^\epsilon}{2}N(A,\delta)^{N+1}\}.$$
One has
$$\delta^\epsilon N(A,\delta)^{n+N+1} \leq N(\Omega,\delta)
\leq N(\Omega',\delta) + N(A,\delta)^n\frac{\delta^\epsilon}{2}N(A,\delta)^{N+1},$$
so that
$$N(\Omega',\delta) \geq \frac{\delta^\epsilon}{2}N(A,\delta)^{N+n+1}.$$
This shows that we may assume without loss of generality that for all $x$ in the projection of $\Omega$ onto the first $n$ coordinates,
$$N(\Omega_x,\delta) \geq \delta^{\epsilon}N(A,\delta)^{N+1}.$$
Now, provided $\epsilon$ is small enough, we may apply Proposition~\ref{outofxi} to the projection of $\Omega$ to the first $n$ coordinates, and we obtain an index $i\in\{1,\dots,N\}$ and $x^0=(x_1^0,\dots,x_n^0)$ such that
$$\xi = \pi_i(x_1^0,\dots,x_n^0) \not\in \Xi.$$
As $N(\Omega_{x^0},\delta) \geq \delta^\epsilon N(A,\delta)^{N+1}$, we may find elements $y_j^0$, $j\not\in\{i-1,i\}$ such that denoting
$$\Omega_1=\{(y_{i-1},y_i) \,|\, (x_1^0,\dots,x_n^0,y_1^0,\dots,y_{i-2}^0,y_{i-1},y_i,y_{i+1}^0,\dots,y_N^0)\in\Omega\},$$
we have
$$N(\Omega_1,\delta) \geq \delta^\epsilon N(A,\delta)^2.$$
By definition of the set $\Xi$ of troublemakers, $\xi\not\in\Xi$ implies that
$$N(\pi_\xi(\Omega_1),\delta) \geq \delta^{-\epsilon}N(A,\delta),$$
where $\pi_\xi:(x,y)\mapsto x\xi y$.
However, it is readily seen that for some elements $u$ and $v$ in $U^L$ (where $L$ is the total length of the word $w$), we have $u\pi_\xi(\Omega_1)v \subset w(\Omega)$, and therefore,
$$N(w(\Omega),\delta) \gg \delta^{-\epsilon} N(A,\delta).$$
\end{proof}

\section{Flattening and dimension increment}
\label{section:flattening}

It is now time to translate the combinatorial results of the previous section into statements about measures, and in turn, about Hausdorff dimension of product sets.

\begin{definition}
A Borel probability measure on the Lie group $G$ is called \emph{$\sigma$-Frostman} if it satisfies, for all $\delta>0$ sufficiently small, and all $x$ in $G$,
$$\mu(B(x,\delta)) \leq \delta^\sigma.$$
\end{definition}

The importance of this definition lies in the following lemma (see Mattila~\cite[Chapter~8]{mattila}).

\begin{lemma}[Frostman's Lemma]
Let $G$ be a Lie group of dimension $d$, and $\sigma\in (0,d)$.
\begin{itemize}
\item Suppose $\mu$ is a $\sigma$-Frostman measure on $G$, and $A$ is a Borel subset of $G$ such that $\mu(A)>0$. Then $\dim_H A\geq \sigma$.
\item Conversely, if $A$ is a Borel subset of $G$ satisfying $\dim_H A >\sigma$, then there exists a $\sigma$-Frostman measure $\mu$ whose support in included in $A$.
\end{itemize}
\end{lemma}

The goal of this section is to prove the following Flattening Lemma, in the spirit of Bourgain-Gamburd \cite[Proposition~1]{bourgaingamburdsu2}.

\begin{lemma}\label{flattening}
Let $G$ be a connected simple Lie group of dimension $d$.
There exists a neighborhood $U$ of the identity in $G$ such that, given $\sigma\in(0,d)$, there exists $\epsilon_1=\epsilon_1(\sigma)>0$ such that the following holds.\\
Suppose $\{a_i\}$ is a family of elements of $U$ satisfying the conclusion of Lemma~\ref{parameters}, and let $w:G^{\times p}\rightarrow G$ be the associated map, as defined in (\ref{ww}).\\
If $\mu$ is a $\sigma$-Frostman finite measure supported on $U$ and $\nu$ is the pushforward of $\mu^{\otimes p}$ under the map $w$, then $\nu*\nu$ is $(\sigma+\epsilon_1)$-Frostman.\\
Moreover, $\epsilon_1$ is bounded away from $0$ if $\sigma$ varies in a compact subset of $(0,d)$.
\end{lemma}

From the flattening lemma, it is easy to prove the results announced in the introduction:

\begin{theorem}
Let $G$ be a connected simple Lie group of dimension $d$. There exists a neighborhood $U$ of the identity in $G$ and a positive integer $k$ such that given $\sigma>0$, there exists $\epsilon=\epsilon(\sigma)>0$ such that if $A\subset U$ is any Borel measurable topologically generating set of Hausdorff dimension $\alpha\in[\sigma,d-\sigma]$ then
$$\dim_H A^k \geq \epsilon+\dim_H A.$$
\end{theorem}
\begin{proof}
Choose a neighborhood $U$ of the identity and $\epsilon_1>0$ such that Lemma~\ref{flattening} holds, and let $\epsilon=\frac{\epsilon_1}{2}$.\\
The set $A$ is topologically generating, so we may choose in a product set $A^s$ a finite collection of elements $\{a_i\}$ satisfying the conclusion of Lemma~\ref{parameters}.\\
By Frostman's Lemma, there exists a Borel probability measure $\mu$ which is $(\alpha-\epsilon)$-Frostman and whose support is included in $A$.\\
Let $\nu$ be the image measure $\nu=w_*(\mu^{\otimes p})$.
All the $a_i$'s are in a product set $A^s$ and the measure $\mu$ is supported on $A$, so there exists an integer $k$ (depending only on $G$) such that $\nu*\nu$ is supported on the product set $A^k$.\\
By Lemma~\ref{flattening}, we know that, provided we have chosen $\epsilon$ small enough, the measure $\nu*\nu$ is $(\alpha+\epsilon)$-Frostman, and this shows that $\dim_H A^k \geq \alpha+\epsilon$.
\end{proof}

As a corollary, we obtain:

\begin{corollary}
Let $G$ be a connected simple real Lie group. Any dense Borel measurable sub-semigroup of $G$ has Hausdorff dimension $0$ or $\dim G$.
\end{corollary}

Before we turn to the proof of Lemma~\ref{flattening}, we record the following elementary lemma.

\begin{lemma}\label{jacobian}
Let $\nu$ be a finite measure on a measurable space $T$, let $U$ be an open subset of $\R^d$, and $\mu$ be a Borel measure on $U$ with square integrable density.
Suppose $w:U\times T\rightarrow \R^d$ is a measurable map such that for each $t$ in $T$, the partial application $w_t:u\mapsto w(u,t)$ is injective and differentiable, with Jacobian $J_{w_t}$.
If $C$ is a positive constant such that,
$$\forall t,u, \quad |J_{w_t}(u)|\geq \frac{1}{C},$$
then the measure $w_*(\mu\otimes\nu)$ has square integrable density, and
$$\|w_*(\mu\otimes\nu)\|_2 \leq C^{\frac{1}{2}}\nu(T)\|\mu\|_2.$$
\end{lemma}
\begin{proof}
Denoting by $f$ the density of $\mu$, it is readily checked that the measure $w_*(\mu\otimes\nu)$ has density $\theta$ given by
$$\theta(z) = \int_T \mathbbm{1}_{\{z\in w_t(U)\}}f(w_t^{-1}(z))|J_{w_t^{-1}}(z)|\,d\nu(t).$$
By Cauchy-Schwarz's inequality, we have
\begin{align*}
\|w_*(\mu\otimes\nu)\|_2^2 & = \int_{\R^d}\left(\int_T \mathbbm{1}_{\{z\in w_t(U)\}}f(w_t^{-1}(z))|J_{w_t^{-1}}(z)|\,d\nu\right)^2\,dz\\
& \leq \nu(T) \int_{\R^d}\int_T \mathbbm{1}_{\{z\in w_t(U)\}}f(w_t^{-1}(z))^2|J_{w_t^{-1}}(z)|^2\,d\nu\,dz
\end{align*}
By assumption, we have for all $t$ and $z$, $|J_{w_t^{-1}}(z)|\leq C$, and therefore, using also Fubini's Theorem and the obvious change of variables,
\begin{align*}
\|w_*(\mu\otimes\nu)\|_2^2 & \leq C\nu(T) \int_{\R^d}\int_T \mathbbm{1}_{\{z\in w_t(U)\}}f(w_t^{-1}(z))^2|J_{w_t^{-1}}(z)|\,d\nu\,dz\\
& = C \nu(T)^2 \|\mu\|_2^2.
\end{align*}
\end{proof}

We will apply the above lemma to the map $w$ defined in (\ref{ww}). By the following lemma, this will be possible, provided we restrict to a suitable neighborhood of the identity.

\begin{lemma}\label{nbhd}
Let $G$ be a simple Lie group.
There exists a neighborhood $U$ of the identity in $G$ and a constant $C$ depending on $G$ only such that the following holds.\\
Suppose $\{a_i\}_{1\leq i\leq k}$  is a finite set of elements of $U$ satisfying the conclusion of Lemma~\ref{parameters}, and let $w:U^p\rightarrow G$ be the corresponding map, defined as in (\ref{ww}).
If $(t_i)_{\substack{1\leq i\leq p\\ i\neq i_0}}$ is any family of elements of $U$, then the partial application
$$w_t: x\mapsto w(t_1,\dots,t_{i_0-1},x,t_{i_0+1},\dots,t_p)$$
is injective on $U$ and its Jacobian satisfies
$$|J_{w_t}| \geq \frac{1}{C}.$$
\end{lemma}
\begin{proof}
Let $\widetilde{w}$ be the map
$$
\begin{array}{ccc}
G^{k+p} & \rightarrow & G\\
(\{a_i\},(x_i)) & \mapsto & w_{\{a_i\}}(x_1,\dots,x_p).
\end{array}
$$
Since $\widetilde{w}$ is a word in the $a_i$'s and $x_i$'s with only positive exponents, its derivative at the identity has the form
$$(n_1 I \,|\, n_2 I \,|\,\dots\,|\, n_{k+p} I),$$
where the $n_i$'s are positive integers.
The lemma easily follows from this observation, by continuity of the derivative of $\widetilde{w}$ and by a quantitative version of the Inverse Function Theorem (see e.g. \cite[Theorem~2.11]{saxceproducttheorem}).
\end{proof}

For any small scale $\delta>0$, we denote by $P_\delta$ the function $\frac{\mathbbm{1}_{B(1,\delta)}}{|B(1,\delta|}$, and if $\mu$ is any Borel measure on the Lie group $G$, we write $\mu_\delta=\mu*P_\delta$.

The proof of Lemma~\ref{flattening} goes by approximating the measure $\mu_\delta$ by dyadic level sets. We say that a collection of sets $\{X_i\}_{i\in I}$ is \emph{essentially disjoint} if for some constant $C$ depending only on the ambient group $G$, any intersection of more than $C$ distinct sets $X_i$ is empty. We will use the following lemma.

\begin{lemma}\label{dyadic}
Let $G$ be a real Lie group and $U$ be a compact neighborhood of the identity in $G$. Suppose $\mu$ is a Borel probability measure on $G$ and $\delta>0$ is some small scale.\\
Then, there exist subsets $A_i$, $0\leq i \ll \log\frac{1}{\delta}$ such that
\begin{enumerate}
\item $\mu_\delta \ll \sum_i 2^i \mathbbm{1}_{A_i} \ll \mu_{4\delta}$
\item Each $A_i$ is an essentially disjoint union of balls of radius $\delta$.
\end{enumerate}
\end{lemma}
\begin{proof}
A proof in the case $G=SU(2)$ is given in \cite{lindenstrausssaxcesu2} and also applies in this more general setting, up to some minor changes.
\end{proof}

\begin{proof}[Proof of Lemma~\ref{flattening}]
Let $\mu$ be a $\sigma$-Frostman probability measure supported on $U$, and assume for a contradiction that for some small $\epsilon>0$ and some arbitrary small ball $B(x,\delta)$, we have
$$\nu*\nu(B(x,\delta)) \geq \delta^{\sigma+\epsilon}.$$
From
$$\nu*\nu(B(x,\delta))\ll \delta^{d}\nu_\delta*\nu_\delta(x)\ll \delta^d \|\nu_\delta\|_2^2,$$
we find
\begin{equation}\label{nularge}
\|\nu_\delta\|_2^2 \gg \delta^{-d+\sigma+\epsilon}.
\end{equation}
Using Lemma~\ref{dyadic}, we approximate $\mu_\delta$ by dyadic level sets:
$$\mu_\delta \ll \sum_i 2^i \mathbbm{1}_{A_i} \ll \mu_{4\delta},$$
each $A_i$ being an essentially disjoint union of balls of radius $\delta$.\\
By inequality (\ref{nularge}),
\begin{align*}
\delta^{\frac{-d+\sigma+\epsilon}{2}} \ll \|\nu_\delta\|_2
& \leq \left\|\sum_{i_1,\dots,i_p} w_*(2^{i_1}\mathbbm{1}_{A_{i_1}}\otimes\dots\otimes 2^{i_p}\mathbbm{1}_{A_{i_p}})\right\|_2\\
& \leq \sum_{i_1,\dots,i_p} \|w_*(2^{i_1}\mathbbm{1}_{A_{i_1}}\otimes\dots\otimes 2^{i_p}\mathbbm{1}_{A_{i_p}})\|_2,
\end{align*}
so there exist indices $i_1,\dots,i_p$ such that
\begin{equation}\label{indices}
\|w_*(2^{i_1}\mathbbm{1}_{A_{i_1}}\otimes\dots\otimes 2^{i_p}\mathbbm{1}_{A_{i_p}})\|_2 \geq \delta^{\frac{-d+\sigma}{2}+O(\epsilon)}.
\end{equation}
Given $\ell$ in $\{1,\dots,p\}$, Lemma~\ref{nbhd} ensures that we may apply Lemma~\ref{jacobian} to the map $w$ and to the measures with density $2^{i_\ell}\mathbbm{1}_{A_{i_\ell}}$ and $\bigotimes_{\ell'\neq\ell}2^{i_{\ell'}}\mathbbm{1}_{A_{i_\ell'}}$ and this yields
\begin{align*}
\delta^{\frac{-d+\sigma}{2}+O(\epsilon)} & \ll \|2^{i_\ell}\mathbbm{1}_{A_{i_\ell}}\|_2 \cdot \|\bigotimes_{\ell\neq\ell'} 2^{i_{\ell'}}\mathbbm{1}_{A_{i_{\ell'}}}\|_1\\
& = 2^{i_\ell}|A_{i_\ell}|^{\frac{1}{2}} \prod_{\ell'\neq\ell} 2^{i_{\ell'}}|A_{i_{\ell'}}|.
\end{align*}
Using also that the definition of the $A_i$'s implies that
$$2^i|A_i|\ll 1 \quad\mbox{and}\quad 2^i|A_i|^{\frac{1}{2}}\ll \|\mu_\delta\|_2,$$
the above forces
$$2^{i_\ell/2} \geq \delta^{\frac{-d+\sigma}{2}+O(\epsilon)} \quad\mbox{and}\quad
\forall \ell'\neq\ell,\ 2^{i_{\ell'}}|A_{i_{\ell'}}|\geq \delta^{O(\epsilon)}.$$
This must hold for each $\ell$, and therefore, for each $\ell$,
\begin{equation}\label{size}
2^{i_\ell}=\delta^{-d+\sigma+O(\epsilon)} \quad\mbox{and}\quad 2^{i_\ell}|A_{i_\ell}|=\delta^{O(\epsilon)}.
\end{equation}
As the set $A_{i_\ell}$ is a union of ball of radius $\delta$, this shows that
$$N(A_{i_\ell},\delta) \gg \delta^{-d}|A_{i_\ell}| \geq \delta^{-\sigma+O(\epsilon)}.$$
Moreover, as the measure $\mu$ is $\sigma$-Frostman, we have, for all $\rho\geq\delta$,
$$2^{i_\ell}|A_{i_\ell}\cap B(x,\rho)| \ll \mu(B(x,4\rho)) \ll \rho^{\sigma},$$
whence
$$N(A_{i_\ell}\cap B(x,\rho),\delta) \ll \delta^{-d}|A_{i_\ell}\cap B(x,\rho)| \leq \rho^\sigma\delta^{-O(\epsilon)}N(A_{i_\ell},\delta).$$
Thus, each $A_{i_\ell}$ is a $(\sigma,O(\epsilon))$-set at scale $\delta$, and therefore, so is 
$$A:=\bigcup_{\ell=1}^p A_{i_\ell}.$$
Now let $\varphi$ be the density function of the measure $w_*(2^{i_1}\mathbbm{1}_{A_{i_1}}\otimes\dots\otimes 2^{i_p}\mathbbm{1}_{A_{i_p}})$.\\
On one hand, by (\ref{indices}), we have
$$\|\varphi\|_2^2 = \delta^{-d+\sigma+O(\epsilon)}.$$
On the other hand, $\mu$ is $\sigma$-Frostman and $\nu$ can be written $\nu_1*\mu$ for some probability measure $\nu_1$, so that $\nu$ is also $\sigma$-Frostman, which implies
$$\|\varphi\|_\infty \leq \delta^{-d+\sigma}.$$
Let
$$E=\{x\in G \,|\, \varphi(x)\geq\frac{\|\varphi\|_2^2}{2}\}.$$
We have
$$\|\varphi\|_2^2  \leq \int_E\varphi^2 + \int_{G\backslash E}\varphi^2
 \leq \|\varphi\|_\infty\int_E\varphi + \frac{\|\varphi\|_2^2}{2}\int_G\varphi
 \leq \|\varphi\|_\infty\int_E\varphi + \frac{\|\varphi\|_2^2}{2}$$
whence
$$\int_E\varphi \geq \frac{\|\varphi\|_2^2}{2\|\varphi\|_\infty} \geq \delta^{O(\epsilon)}.$$
Letting $\Omega$ be the inverse image $w^{-1}(E)$, the above inequality certainly implies that
$$\mu^{\otimes k}(\Omega) \geq \delta^{O(\epsilon)},$$
which, by the fact that $\mu$ is $\sigma$-Frostman, shows that
$$N(\Omega,\delta) \geq \delta^{-k\sigma+O(\epsilon)} \geq \delta^{O(\epsilon)}N(A,\delta)^k.$$
To obtain a contradiction, we will bound the size of $w(\Omega)=E$ using that $\varphi$ takes large values on that set.
First observe that isolating the last letter of $w$ -- in (\ref{ww}), the letter $y_N$ -- allows us to write $\varphi$ as a convolution
$$\varphi = \varphi_1 * (2^{i_p}\mathbbm{1}_{A_p}).$$
Then, as $A_p$ is a union of balls of radius $\delta$, we have $\mathbbm{1}_{A_p} \ll \mathbbm{1}_{A_p}*P_{\frac{\delta}{2}}$ and therefore,
$$\varphi \ll \varphi*P_{\frac{\delta}{2}}.$$
In particular, for each $x$ in $E$,
$$\frac{\|\varphi\|_2^2}{2} \leq \varphi(x) \ll \delta^{-d}\int_{B(x,\frac{\delta}{2})}\varphi,$$
and summing this inequality for $x$ in a maximal $\delta$-separated set in $E$, we find
$$\frac{\|\varphi\|_2^2}{2}N(E,\delta) \ll \delta^{-d}\int \varphi
\leq \delta^{-d}.$$
Thus,
$$N(w(\Omega),\delta) \leq \delta^{-\sigma-O(\epsilon)} \leq \delta^{-O(\epsilon)}N(A,\delta),$$
which contradicts Corollary~\ref{finalcombinatorial}, provided we have chosen $\epsilon$ small enough.
\end{proof}

\bibliographystyle{plain}
\bibliography{bibliography}

\end{document}